\documentclass [a4paper]{article}

\usepackage{amsmath}
\usepackage{amsthm}
\usepackage{amssymb}

\usepackage{multirow}

\newtheorem{theorem}{Theorem}

\newtheorem{lemma}[theorem]{Lemma}

\usepackage[section]{algorithm}
\usepackage{algorithmic}
\algsetup{linenodelimiter=.}
\algsetup{indent=2em}
\floatname{algorithm}{\textsf{\textbf{Algorithm}}}

\usepackage{eqparbox}

\newcommand{\vc}{\mathbf}
\newcommand{\norm}[1]{\left\lVert#1\right\rVert}

\begin{document}

\title{Optimization techniques for  multivariate least trimmed absolute deviation estimation}

\author{
 G. Zioutas, C. Chatzinakos, L. Pitsoulis\thanks{work of this author was conducted at National Research University Higher School of Economics and supported by RSF grant 14-41-00039}\\
  Department of Electrical and Computer Engineering \\
  University of Thessaloniki, Greece
  \and
  T.D. Nguyen\\
  Southampton Management School\\
  United Kingdom
}

\maketitle

\begin{abstract}
Given a dataset an outlier can be defined as an observation that it is unlikely to follow the statistical properties of the majority of the data. 
Computation of the location estimate of is fundamental in data analysis, and it is well known in statistics that classical methods, such as taking the sample
average, can be greatly affected by the presence of outliers in the data. Using the median instead of the mean can partially resolve this issue but not completely.
For the univariate case, a robust version of the median is the Least Trimmed Absolute Deviation (LTAD) robust estimator introduced in~\cite{Tableman1994}, 
which has desirable asymptotic properties such as robustness, consistently, high breakdown and normality. There are different generalizations of the LTAD for
multivariate data, depending on the choice of norm. In~\cite{ChaPitZiou:2015} we present such a generalization using the Euclidean norm and propose a 
solution technique for the resulting combinatorial optimization problem, based on a necessary condition, that results in a highly convergent local search algorithm. 
In this subsequent work we use the $L^1$ norm to generalize the LTAD to higher dimensions, and show that the resulting mixed integer programming problem
has an integral relaxation, after applying an appropriate data transformation. Moreover, we utilize the structure of the problem to show that the resulting LP's can 
be solved efficiently using a subgradient optimization approach. The robust statistical properties of the proposed estimator are verified by extensive computational results.

\vspace*{.25in}
\end{abstract}

\section{Introduction}

The sample average and standard deviation are the classical estimators of the location and the
scale parameters of a statistical distribution. It is well-known that these classical estimators, although being optimal under normality assumptions, are extremely sensitive to 
the presence of outliers in the data;  a small proportion of outliers in the data can have a large distorting effect on the sample mean and covariance. 
Robust statistics is concerned with the development of methods for computing estimators that  are justifiably resistant to the presence of outliers in the data. 
The focus of this work is to estimate the unknown location 
parameter $\vc{m}$ of a family of distributions $F_\vc{m}$ given some observational data which is contaminated with an unknown number  of outliers.

Detecting outliers and unusual data structures is one of the main problems in statistical data analysis since this occurs in many different application domains. 
One such application is in large scale complex networks, where the graph data may arrive in streams~\cite{Aggar:2013, Aggar:2011}. Given that the degrees in
social networks typically follow a power law distribution, it is of importance to identify outlier nodes which do not follow the degree distribution of the majority
of the nodes, since these nodes may affect the computation of several graph characteristics such as community detection, clustering coefficient etc.

Projection pursuit is one of the typical approaches for outlier detection. The idea is to repeatedly project the multivariate data into the univariate space since 
univariate outlier detection is much simpler to handle by applying order statistics and visualization. 
 Such methods are usually computationally intensive, but they are particularly useful for high-dimensional  data with small sample size. 
One such technique is the principal components analysis in~\cite{Filzmoser20081694}.
Outlier detection can also be done based on estimations of the covariance matrix. The idea is to use estimated covariance structure in order to find a distance, usually 
the well-known Mahalanobis distance, from each observation to the center of the data cloud. One such method is the Minimum Covariance Determinant (MCD) 
introduced in \cite{Rousseeuw1985} and in \cite{Rousseeuw1998}.
Desirable properties for an estimator include high breakdown values, high efficiency, and fast computation. One famous robust location estimator is the multivariate 
Least Euclidean Distance (LED) as studied in \cite{Hettmansperger2002}. Other methods for robust location estimation include the transformation median 
(\cite{chakraborty1998}) and the Oja Multivariate half samples Median (HOMM) \cite{Oja1983}. The corresponding univariate cases of the half samples (HOMM) and 
MCD are the Least trimmed Absolute Deviation (LTAD) and Least Trimmed Squared (LTS) estimators respectively.

The LTAD robust estimator for  univariate data was introduced in~\cite{Tableman1994}, where it was shown that it has desirable asymptotic properties such as robustness, 
consistently, high breakdown and normality. Moreover,  in~\cite{Tableman1994} the author also presents an algorithm to efficiently compute the LTAD
in $\mathcal{O}(n\log n)$ time. However, these methods do not generalize to higher dimensions. 
In~\cite{ChaPitZiou:2015} the LTAD is generalized to handle  multivariate data using the Euclidean norm, and the resulting combinatorial optimization problem
is solved by an approximate fixed-point like iterative procedure. 
Computational experiments in~\cite{ChaPitZiou:2015} on both real and artificial data  indicate that the proposed method efficiently identifies both location and scatter 
outliers in varying dimensions and high degree of contamination. 
In this work we extend the results in~\cite{ChaPitZiou:2015}, and present a different generalization of LTAD which is based on the $L^1$ norm. It is shown
that  the linear programming relaxation of the resulting mixed integer program is integral, after applying an appropriate equidistance data transformation. 
This implies that the LTAD can be computed as a series of linear programs, which can  be solved efficiently using a subgradient optimization approach.

The rest of this paper is structured as follows. In section~\ref{sec_LTAD} we present the generalized LTAD for multivariate data using the $L^1$ norm while  it's
mixed integer programming formulation is given in section~\ref{sec_MIP_LTAD}. The integrality of the relaxation is presented in section~\ref{sec_LP_LTAD}, along
with the procedure to perform the data transformation. The subgradient optimization approach for solving the related linear programs is presented in 
section~\ref{sec_LPs} and in section~\ref{sec_computations} we perform computational experiments in real and simulated data to verify the performance
of the proposed estimator. Conclusions are given in~\ref{sec_conclusions}.

\section{Least trimmed absolute deviation estimator} \label{sec_LTAD}
Given a sample of $n$ univariate observations $X_{n} = \{ x_{1}, x_{2}, \ldots, x_{n} \}$ where $x_i \in \mathbb{R}, i=1,\ldots, n$ we can
state the well known location parameter {\em median} as follows: 
\begin{eqnarray}\label{eq_median}
m(X_{n}) =  \arg\min_{m} & &  \sum_{i=1}^{n} | x_{i} - m | 
\end{eqnarray}
Letting $(\cdot)$ to denote the  order of the data points, a trivial solution to (\ref{eq_median}) is 
\[
x_{(1)} \leq x_{(2)} \leq \cdots \leq x_{(n)}
\]
then
\[
m(X_{n}) = 
\left\{ 
\begin{array}{ll}
x_{\left( \frac{n+1}{2} \right) } & \textrm{if  $n$ is odd,} \\
\frac{1}{2} \left(  x_{\left( \frac{n}{2} \right) }  + x_{\left( \frac{n+1}{2} \right) } \right)& \textrm{if  $n$ is even.}
\end{array} \right.
\]
The {\em mean} is another example of a location estimator for a data set $X$, as well as the {\em least median of squares}~\cite{rousseeuw1} which is the 
midpoint of the subset that contains half of the observations with the smallest range. 
Three statistically desirable properties of an estimator are  {\em equivariance}, {\em monotonicity} and {\em 50\% breakdown point}. 
Equivariance implies that if the data points are scaled and shifted then the value of the estimator will change accordingly, while monotonicity implies that the estimator cannot decrease in value if an observation increases. An estimator has 50\% breakdown point if its value will be bounded for any arbitrary change
of less than half of the observations. Basset~\cite{Bassett:1991} has proven that the median is the only estimator that satisfies all three properties.

If we make the assumption that  $n-h$ of observations are outliers, where $h > \lceil n/2 \rceil$, we can define a robust version of the 
median which will be called {\em least trimmed absolute deviations} (LTAD) estimator, defined by the following problem: 
\begin{eqnarray}\label{eq_ltad}
m(X_{n},h) =  \arg\min_{m,T} & &  \sum_{x\in T}   | x - m |  \\ \nonumber
                \mbox{s.t.}     & &   |T|=h \\ \nonumber
                                      &  & T\subseteq X_{n}
\end{eqnarray}
which implies that we have  to find that subset $T$ of $h$ observations out of $n$ which have the least median value. In order to satisfy the high 
breakdown property, the 
value of $h$ is set to $\lceil n/2 \rceil$.  Solving  (\ref{eq_ltad}) by complete enumeration would require the computation of the median 
for all possible ${n \choose h}$ subsets $T\subseteq X_{n}$ and choosing the one with the minimum value, which is computationally 
infeasible even for moderate values of $n$ and $p$. 
The LTAD was introduced by Tableman~\cite{Tableman1994} for fixed $h=\lceil n/2 \rceil$, where in addition to showing favorable theoretical properties the 
author also provides a simple procedure for its computation based on the observation that 
the solution to \eqref{eq_ltad} will be the median of one of the following $n-h$ contiguous subsets 
\[
\{x_{(1)}, \ldots, x_{(h)} \},  \{x_{(2)}, \ldots, x_{(h+1)} \}, \ldots, \{x_{(n-h)}, \ldots, x_{(n)} \}.
\]
Therefore, it suffices to compute the $n-h$ median values for the above subsets, a process that will require $\mathcal{O}(n\log n)$ time to order
the data points according to increasing value.

Consider now the multidimensional version of the LTAD defined in \eqref{eq_ltad}, 
where we have $p$-variate observations $X_{n} = \{\vc{x}_1, \ldots, \vc{x}_{n} \}$ with
$\vc{x}_{i} \in \mathbb{R}^{p}, i=1,\ldots, n$. Moreover,  without loss of generality we can assume that the observations  are rescalled. 
The multivariate LTAD is defined as
\begin{eqnarray} 
\displaystyle \text{LTAD}:\vc{m}(X_{n},h) =  \arg\min_{\vc{m},T} & &  \sum_{\vc{x}\in T}   \| \vc{x} - \vc{m} \|_{1}   \label{eq_ltad_p} \\ 
                \mbox{s.t.}     & &   |T|=h \nonumber \\ 
                                      &  & T\subseteq X_{n} \nonumber
\end{eqnarray}
where $\| \cdot \|_{1}$ stands for the one norm i.e. $\norm{\vc{x}-\vc{m}}_1 = \sum_{i=1}^p |x_ i - m_i|$. 
For the ease of exposition in the rest of the paper, we refer to both the univariate and the multivariate LTAD as the {\em LTAD problem}.

The LTAD problem can be approximated by an iterative algorithm similar to the procedure described in~\cite{ChaPitZiou:2015} for 
solving the related {\em least trimmed Euclidean distances} (LTED) estimator, which is defined as the LTAD in \eqref{eq_ltad_p} with the only 
exception that the euclidean norm is used instead of the one norm.  The same necessary optimality condition that is proved for
LTED in~\cite{ChaPitZiou:2015} can also be proved for the LTAD, which leads to a highly convergent heuristic algorithm.  
However, although this algorithm is very fast, it almost always converges to a local optimum of unknown quality.

In this paper we present a different solution method for the LTAD, by approximating its natural mixed integer nonlinear programming formulation
with with a mixed integer linear program whose linear programming relaxation is integral. We also develop specialized efficient solution methods for the
resulting linear program, since the iterative nature of the proposed method requires multiple calls for solving them.

\section{Mixed integer programming formulation}\label{sec_MIP_LTAD}
The LTAD estimate \eqref{eq_ltad_p} can be equivalently stated as the following mixed integer nonlinear programming problem 
\begin{eqnarray}\label{eq_ltad_ip_1}
\displaystyle \text{MINLP-LTAD}: ~\min_{\vc{w}, \vc{m}} & &  \sum_{i=1}^{n}w_{i} \norm{\vc{x}_{i}-\vc{m}}_1 \\ \nonumber
                \mbox{s.t.}     & &  \sum_{i=1}^{n} w_{i} = h \\ \nonumber
                                       & &  w_i \in \{0,1\}, i = 1,\ldots,n. 
\end{eqnarray}
where the zero-one weights  $\vc{w}=(w_{1}, \ldots, w_{n})$ indicate whether observation $i$ is an outlier ($w_i=0$) or a good observation ($w_i=1$). 
For any feasible tuple $(\vc{w}, \vc{m})$ to \eqref{eq_ltad_ip_1}, let $\vc{x}_{(i)}$ denote the vector  $\vc{x} \in X_{n}$ with the $i$-th smallest $\norm{\vc{x}-\vc{m}}_1$
value, and $w_{(i)}$ its corresponding weight. We can now write \eqref{eq_ltad_ip_1} as follows
\begin{eqnarray*}
\sum_{i=1}^{n}w_{i}\norm{\vc{x}_{i}-\vc{m}}_1 &=& \sum_{i=1}^{h}\norm{\vc{x}_{(i)}-\vc{m}}_1 \\
&=& \sum_{i=1}^{h}\norm{w_{(i)} \vc{x}_{(i)}-\vc{m}}_1\\
&=& \sum_{i=1}^{n}\norm{w_i \vc{x}_{i}-\vc{m}}_1 - (n-h) \norm{\vc{m}}_1.
\end{eqnarray*}
since $w(i) = 1$ for all $i=1,\ldots, h$. 
Observe that as $\vc{m}$ approaches zero, then $\sum_{i=1}^{n}\norm{w_i \vc{x}_{i}-\vc{m}}_1$ approaches
$\sum_{i=1}^{n} w_{i}\norm{\vc{x}_{i}-\vc{m}}_1$, thus, for small values of $\vc{m}$ problem \eqref{eq_ltad_ip_1} can be approximated by the following 
\begin{eqnarray*}
\min_{\vc{w}, \vc{m}} & &  \sum_{i=1}^{n} \norm{ w_{i}\vc{x}_{i}-\vc{m}}_1 \\ 
                \mbox{s.t.}     & &  \sum_{i=1}^{n} w_{i} = h \\
                                       & &  w_i \in \{0,1\}, i = 1,\ldots,n. 
\end{eqnarray*}
which is equivalent to the following mixed integer linear program
\begin{eqnarray}
\displaystyle \text{MILP-LTAD}: ~\min_{\vc{w},\vc{m}} &  &  \sum_{i=1}^{n} \sum_{j=1}^{p}  {d}_{ij} \label{eq_ltad_ip_3} \\
                \mbox{s.t.}      & & \sum_{i=1}^{n} w_{i} = h   \nonumber \\
                                       & & w_i x_{ij} - m_j - d_{ij} \leq 0, \;\; i = 1,\ldots,n,  j = 1,\ldots,p \nonumber \\
     		                       & & -w_i x_{ij} + m_j - d_{ij} \leq 0,  \;\; i = 1,\ldots,n,  j = 1,\ldots,p \nonumber \\
                                        & & \vc{w} \in \{0,1\}^{n} \nonumber
\end{eqnarray}
where $\vc{D}$ is an $n\times p$ matrix with values $d_{ij} = |w_i x_{ij}-m_j|$, and $X=(x_{ij})$ is the $n\times p$ observations matrix whose rows are
$\vc{x}^{T}_{i}$ for $i=1,\ldots,n$. 

There are two issues with approximating \eqref{eq_ltad_ip_1} with \eqref{eq_ltad_ip_3}. First of all, we need to ensure that MILP-LTAD is a good 
approximation of the MINLP-LTAD. Secondly, we need to be able to solve \eqref{eq_ltad_ip_3} efficiently. We will resolve the first issue by iteratively 
transforming the data such that the optimal $\vc{m}$ approaches zero. For the second issue we will show that the resulting mixed integer linear programming
problem is equivalent to a linear programming problem under certain assumptions.

\section{Data transformation} \label{sec_LP_LTAD}

Denote with LP-LTAD the linear programming relaxation of \eqref{eq_ltad_ip_3} where $\vc{w} \in [0,1]^{n} $, 
Consider the linear programming relaxation of \eqref{eq_ltad_ip_3}
\begin{eqnarray}
\displaystyle \text{LP-LTAD}: ~\min_{\vc{w},\vc{m}} &  &  \sum_{i=1}^{n} \sum_{j=1}^{p}  {d}_{ij} \label{eq_ltad_ip_4} \\
                \mbox{s.t.}      & & \sum_{i=1}^{n} w_{i} = h   \nonumber \\
                                       & & w_i x_{ij} - m_j - d_{ij} \leq 0, \;\; i = 1,\ldots,n,  j = 1,\ldots,p \nonumber \\
     		                       & & -w_i x_{ij} + m_j - d_{ij} \leq 0,  \;\; i = 1,\ldots,n,  j = 1,\ldots,p \nonumber \\
                                        & & \vc{w} \in [0,1]^{n} \nonumber
\end{eqnarray}
Let $(\vc{w}^*_{LP},\vc{m}^*_{LP})$ be the optimum solution of LP-LTAD. If $\vc{w}^*_{LP}$ is integer, then this LP solution is also an 
optimal solution of \eqref{eq_ltad_ip_3}.

We show next, that if in the linear programming  optimum solution $\vc{m}^*_{LP}$ is equal to zero, then $(\vc{w}^{*}_{LP},\vc{m}^*_{LP})$ is optimal for the
MILP-LTAD; that is, we can solve the linear programming relaxation and use it to obtain an optimal solution for the MILP in \eqref{eq_ltad_ip_3}. 

\begin{lemma}\label{lemma1}
For any $\vc{x}$, if $\vc{m}^*_{LP}=\vc{0}$, then $(\vc{w}_{LP}^{*},\vc{m}^*_{LP})$ is an optimal solution of MILP-LTAD.
\end{lemma}
\begin{proof}
Let $f^*_{LP}$ and $f^*_{MILP}$ be the optimal solutions of LP-LTAD and MILP-LTAD respectively. If $\vc{m}^*_{LP}=\vc{0}$ then
\[
f^*_{LP} = \sum_{i=1}^{n} \| w_{i}^{*} \vc{x}_{i} - \vc{m}_{LP}^{*} \|_{1} = \sum_{i=1}^{n} w_{i}^{*} \|x_{i}\|_{1} = \sum_{i=1}^{h} \|x_{(i)}\|_{1}
\]
which implies that $w^{*}_{i} = 1$ if $i=(i)$ and zero otherwise; or equivalently $w_{LP}^{*} \in \{0,1\}^{n}$. Thus, $(\vc{w}_{LP}^{*},\vc{m}^*_{LP})$
is feasible to MILP-LTAD and $f^*_{LP} = f^*_{MILP}$.
\end{proof}

Lemma~\ref{lemma1} implies that if we could transform the data in such a way that $\mu_{LP}^*$ gets closer to zero, then we can just solve the LP problem to obtain an approximated solution for the LTAD. 
This leads  to the procedure described in Algorithm~\ref{alg:LTAD}, where the data is iteratively transformed such that its median value is less than some small value $\epsilon$. 
\begin{algorithm}
\caption{Linear Programming Approach for Solving LTAD}  
\label{alg:LTAD}
    \begin{algorithmic}[1]
      \REQUIRE {data $X_{n} = \{\vc{x}_1, \ldots, \vc{x}_n\}$, coverage $h$, accuracy $\epsilon$}
      \ENSURE  {a set of $h$ data points of $X_n$ as indicated by the characteristic vector $\vc{w}_{LP}$} 
      \vspace{0.05in}
      \WHILE {TRUE}
         \STATE $(\vc{w}_{LP}^{*}, \vc{m}_{LP}^{*})$ = LP-LTAD($X_{n},h$)
          \IF {$\norm{\vc{m}^*_{LP}}  < \epsilon$}
             \RETURN $\vc{w}_{LP}^{*}$
          \ELSE
              \STATE $\vc{x}_{i} :=  \vc{x}_{i} -\vc{m}^*_{LP}, ~\forall~ i=1,...,n$
           \ENDIF
       \ENDWHILE
 \end{algorithmic}
\end{algorithm}

\section{Solution of the LP relaxation}\label{sec_LPs}
In Algorithm~\ref{alg:LTAD} we have to solve the associated linear programming problem in each iteration, until $\vc{m}_{LP}$ converges to a value smaller than $\epsilon$.  
The LP   as defined in \eqref{eq_ltad_ip_4}, has $(np+n+p)$ decision variables and $(2np+2n+1)$ 
constraints and can be solved efficiently for relatively small $(n,p)$. 
However, for large values of $n,p$ (e.g. $n=10000$ and $p = 100$), the problem has a million decision variables and two million constraints. Although this is still solvable, we need to
find an efficient solution method since there will be multiple calls to the solution of this LP. 
 In what follows we will exploit the special structure of the problem to develop such a method. 

Given $\vc{w}$, let $m_j(\vc{w})$ be the corresponding median of vector $\{w_i x_{ij} : i = 1,\ldots,n\}$. 
This means $\vc{m}(\vc{w}) = (m_1(\vc{w}),\ldots,m_p(\vc{w}))$  is an optimal solution of \eqref{eq_ltad_ip_4} for a fixed $\vc{w}$ . 
Letting $f(\vc{w}) = \sum_{i=1}^n \norm{w_i \vc{x}_i - \vc{m}(\vc{w})}_1$ we can write \eqref{eq_ltad_ip_4} as
\begin{eqnarray*}
\min_{\vc{w}} && f(\vc{w}),\\
s.t. && \sum_{i=1}^n w_i = h, \\
&&  0 \leq w_i \leq 1.
\end{eqnarray*}
Here we have transformed the original problem into a new optimization problem in $\mathbb{R}^n$ that has a nice constraint set. However, the objective function is non-linear. 
Rewriting $f(\vc{w})$ we have 
\begin{eqnarray*}
f(\vc{w}) &=& \sum_{i=1}^n \sum_{j=1}^p  |w_i x_{ij} - m_j(\vc{w})|\\
    &=& \sum_{j=1}^p  \sum_{i=1}^n |w_i x_{ij} - m_j(\vc{w})|\\
    &=& \sum_{j=1}^p  \sum_{i=1}^n |w_i x_{ij} - median(\{w_i x_{ij}:i=1,\ldots,n\})|\\
    &=& \sum_{j=1}^p \left[ \sum_{i=1}^n w_i x_{ij} - 2 \min_{|S| = n/2} \sum_{k \in S} w_k x_{ik} \right],\\
\end{eqnarray*}
which is a piece-wise convex function because it is the sum of a linear function and the maximum of linear functions. 
This in turn implies that we can solve the problem using a projected subgradient method, which is  shown in Algorithm~\ref{alg:subgradient}.
\begin{algorithm}
\caption{Subgradient method for solving LP-LTAD}  
\label{alg:subgradient}
    \begin{algorithmic}[1]
      \REQUIRE {initial $\vc{w}^0$,  $\alpha = 1$ and tolerance $\epsilon$}
      \ENSURE  {solution $\vc{w}^{*}$ to LP-LTAD} 
      \vspace{0.05in}
      \STATE $k=0$
      \WHILE {TRUE}
          \STATE Find subgradient $\vc{d}^k = \nabla f(\vc{w})$ and set $\bar{\vc{w}} = \vc{w}^k - \alpha \vc{d}^k$
          \STATE Find the projection $\vc{w}^{k+1}$ of $\bar{\vc{w}}$ on the polyhedron $F = \{\vc{w} : \sum_{i=1}^n w_i = h, 0 \leq w_i \leq 1\}$
          \IF {$\norm{\vc{w}^{k+1}-\vc{w}^k} < \epsilon$}
             \RETURN $\vc{w}^{k+1}$ 
          \ELSE
              \STATE $k=k+1$ 
           \ENDIF
       \ENDWHILE
 \end{algorithmic}
\end{algorithm}

In order to apply the projected subgradient method depicted in Algorithm~\ref{alg:subgradient}, we need to resolve the following three issues: (a) find good initial starting $\vc{w}^{0}$, 
(b) compute the subgradients, and (c) perform efficient projection onto the polyhedron. 
Methods for resolving these issues will be presented in the next subsections.

\subsection{Finding good initial starting point $\vc{w}^0$}
We will find an initial starting point $\vc{w}^0$ by finding a local optimal solution $(\vc{w},\vc{m})$ for the original problem. 
The idea is to start with an arbitrary initial solution $(\vc{w}^0,\vc{m}^0)$, set $k=0$, and repeat the following steps:
\begin{itemize}
    \item[(a)] Fix $\vc{m}=\vc{m}^{k}$ and solve for the corresponding optimal $\vc{w}^{k+1}$
    \item[(b)] If $\norm{\vc{w}^{k+1}-\vc{w}^k} < \epsilon$, return $\vc{w}^{k+1}$ and terminate the procedure.
                  Otherwise, fix $\vc{w}=\vc{w}^{k+1}$ and solve for the corresponding optimal $\vc{m}^{k+1}$. Set $k=k+1$ and go back to step (a).
\end{itemize}
In step (b), for each fixed $\vc{w}$ finding the corresponding optimal $\vc{m}$ is easy, as $m_j$ can be set as the median of  
$\{w_i x_{ij} : i = 1,\ldots,n\}$. Finding the optimal $\vc{w}$ for each fixed $\vc{m}$ is non-trivial unless we reformulate it as an LP, 
but this will be computationally inefficient for large $(n,p)$. A more efficient method is to use a subgradient method to solve the Lagrangian 
dual problem by noticing that the problem has only one linking constraint $\sum_{i=1}^n w_i = k$.
Specifically, for a given $\vc{m}$, we need to solve
\begin{eqnarray}
\min_{\vc{w}} && \sum_{i=1}^n \norm{w_i \vc{x}_i - \vc{m}}_1,\nonumber\\
s.t. && \sum_{i=1}^n w_i = k, \label{constr:Init_w}\\
&&  0 \leq w_i \leq 1.\nonumber
\end{eqnarray}
Let $\delta$ be the Lagrangian multiplier of the equality constraint (\ref{constr:Init_w}). The Lagrangian dual problem is
\begin{eqnarray*}
\max_{\delta}&&  \left( k\delta + \min_{0 \leq w_{i} \leq 1} ~~ \sum_{i=1}^n \norm{w_i \vc{x}_i - \vc{m}}_1 - \delta \vc{w}^{T} \vc{e} \right),\nonumber\\
\end{eqnarray*}
which can be further simplified as
\begin{eqnarray*}
\max_{\delta}~~\left( k\delta + \sum_{i=1}^n \left[ \min_{0 \leq w_i \leq 1} ~~ |w_i \vc{x}_i - \vc{m}| - \delta w_i \right] \right),
\end{eqnarray*}
For each fixed $\delta$, the inner problem has a closed form solution for $w_i$ by noticing that the function $g_i(w_i) = \norm{w_i \vc{x}_i - \mu}_1 - \delta w_i$ is piece-wise 
convex with at most $(p+1)$ pieces that join each other at $m_j/x_{ij}$. This means we can find the optimal $w_i$ by simply comparing the objective values at those joints that 
belong to $[0,1]$. As the inner problem has a closed-form solution and as the outer problem has only a single variable $\delta$, the Lagrangian dual problem can be solved 
very efficiently. Summarizing, we can repeatedly find improving $(\vc{w}^k,\vc{m}^k)$ and stop the process at a local optimal solution of \eqref{eq_ltad_ip_4}.

For finding subgradients, we notice that
\begin{eqnarray*}
f(\vc{w}) &=& \sum_{i=1}^n \norm{w_i \vc{x}_i - \vc{m}(\vc{w})} = \sum_{j=1}^p \underbrace{\sum_{i=1}^n |w_i x_{ij} - m_j(\vc{w})|}_{f_j(\vc{w})},
\end{eqnarray*}
where
\begin{eqnarray*}
\frac{\partial f}{\partial w_k} &=& \sum_{j=1}^p \frac{\partial f_j}{\partial w_k}= \sum_{j=1}^p x_{kj} sign(w_k x_{kj} - m_j(\vc{w})).
\end{eqnarray*}

\subsection{Finding a projection}
The projection of a point on a polyhedron can be found by solving a convex quadratic optimization problem. 
However, this is not a computationally  efficient way and we need to find an alternative by exploiting the special constraint set for $\vc{w}$. 
Notice that this includes only one hyperplane $\sum_{i=1}^n w_i = k$ and a set of box constraints. Thus, the projection of any point $\vc{w}$ into this polyhedron can be found through two steps:
\begin{itemize}
    \item[(a)] Finding the projection $\vc{w}_P$ of $\vc{w}$ onto the plane $\sum_{i=1}^n w_i = k$ which has a closed form solution.
    \item[(b)] Finding the projection $\vc{w}_B$ of $\vc{w}_P$ into the box constraints. This is simply done by  setting:
\[
w_{B,i} = \begin{cases}
 w_{P,i} & \text{ if } 0 \leq w_{P,i} \leq 1,\\
 0 &  \text{ if } w_{P,i} < 0,\\
 1 &  \text{ if } w_{P,i} > 1.
<\end{cases}
\]
\end{itemize}

\section{Computational experiments} \label{sec_computations}
In this section  the performances of LP-LTAD and MILP-LTAD estimators are compared against the performance  a heuristic  iterative algorithm for solving 
the LTAD based on the Algorithm 2.1 given in~\cite{ChaPitZiou:2015}.
The solutions of the associated problems \eqref{eq_ltad_ip_4}  and \eqref{eq_ltad_ip_3},  were computed using the solver FortMP/QMIP which is 
a Fortran code provided by \cite{MiGuEll:03}. The computation of the LTAD solutions were obtained by a MATLAB
implementation of the algorithm in  \cite{ChaPitZiou:2015}.

\subsection{Empirical efficiency}

Most of the robust estimators in the literature choose a priori a coverage of $h = \frac{n}{2}$, which yields a clean subsample of minimum size.
However, if there are fewer outliers in the sample than half of the observations, then information will be discarded when calculating robust estimates based on this. 
As a consequence these estimates suffer from low efficiency. One solution to this problem is to adapt $h$, resulting in more efficient estimators which have 
lower breakdown points. In other words, most robust estimators have to deal with this robustness versus efficiency trade off.

A typical procedure for empirically evaluating the efficiency of robust estimators,  is to apply the estimators on a clean data set and compare their performance. 
We conducted a simulation with a  a sample data set of 100 observations that follow the standard normal distribution with $N(\vc{0},\vc{I})$. 
After 100 replications, the comparison criterion is the average classical center estimate, or median,  for the different coverage sizes  
$h=50\%, 60\%,70\%$ and $80\%$ as it is demonstrated in Table \ref{tltad_21}.
\begin{table}[h!]
\centering
\begin{tabular}{c c c c c c c c c c}
& coverage $h$ & LP-LTAD & MILP-LTAD & LTAD \\ \hline
\multirow
{3}{*}
& 50$\%$ & 0.0000  & 0.0000  & 0.0407 \\ 
& 60$\%$ & 0.0001  & 0.0002  & 0.0327  \\ 
& 70$\%$ & 0.0009  & 0.0009  & 0.0301 \\ 
& 80$\%$ & 0.0011  & 0.0012  & 0.0294 \\ \hline
\end{tabular}
\caption{Median estimate for data set of normal distribution, $N(0,1)$} \label{tltad_21}
\end{table}
We observe from Table~\ref{tltad_21}, that the proposed robust location estimator  LP-LTAD  improves significantly with respect to efficiency, as  the coverage 
$h$ decreases. For the smallest coverage $h=0.5$, which means that $50\%$ of the observations are considered as clean, it yields a location estimate which is
the true value $\vc{m}=\vc{0}$. On the contrary, the ordinary LTAD losses in efficiency as the coverage decreases, resulting in a biased location estimate.

\subsection{Real data}
For our computational experiment involving real data, we will use the data by~\cite{Stefan} for the L1-type estimator. 
The data set originates from the \textit{The Data and Story Library}\footnote{\tt http://lib.stat.\-cmu.edu/\-DASL/\-Stories/\-Forbes500CompaniesSales.html},
which contains facts regarding 79 companies selected from the Forbes 500 list of 1986. 
We consider  the following six variables: {\em assets} (amount of assets in the company in millions), {\em sales} (amount of sales in millions), 
{\em market-value} (market-value of the company in millions), {\em profits} (profits in millions), {\em cash-flow} (cash-flow in millions) 
and {\em employees} (number of employees in thousands). 
We apply the L1-type, LP-LTAD, MILP-LTAD, and LTAD estimators to find an estimate of location. Table~\ref{tltad_32} compares the location estimates 
with the empirical mean. Clearly, there are large differences between the locations estimates of the above estimators with the empirical mean. 
The empirical means are much higher than all estimators. These differences are caused by the presence of outliers in the data set, which greatly affect the
mean statistic. We can see that all estimators perform comparably with respect to obtaining a robust location estimate.

\begin{table}[h!]
\centering
{\footnotesize
\begin{tabular}{cc c c c c c c c c}
&  & Assets & Sales & Market-value & Profits & Cash-flow & Employees \\ \hline
& Empirical mean & 5940.53 & 4178.29 & 3269.75 & 209.84 & 400.93 & 37.60 \\
& L1-type mean  & 2679.33 &  1757.50 & 1099.14 & 89.90 & 164.10 & 15.63 \\
& LP-LTAD       & 2677.21 & 1752.72 & 1096.78 & 88.71 & 164.02 & 15.19 \\
& MILP-LTAD      & 2680.35 & 1754.13 & 1099.24 & 90.18 & 165.09 & 16.01 \\
& LTAD      & 2681.25 & 1756.11 & 1098.21 & 90.03 & 164.97 & 16.07 \\ \hline
\end{tabular}
}
\caption{Estimate for the location of the Forbes data set}\label{tltad_32}
\end{table}

\subsection{Simulation results}

To study the finite-sample robustness and efficiency of the three robust location estimates, we perform simulations with contaminated data sets. 
In each simulation we generate 100 data sets based on a  multivariate normal distribution $N_p(\vc{0},\vc{I})$ with $p=1,2,3,5$ and sample sizes $n=50,100$. 
To generate contaminated data sets we replaced $\epsilon \in \{20\%, 40\%\}$ of the data $\{\vc{x}_1,\ldots,\vc{x}_n\}$ with outliers from a multivariate normal 
distribution, $N_p(3.3,0.3{^2})$.
For each of the aforementioned data sets we obtain the LP-LTAD and MILP-LTAD local estimate $\hat{\vc{m}}$ by computing the solution of the 
problems formulated in \eqref{eq_ltad_ip_4}  and \eqref{eq_ltad_ip_3}, respectively. 
After 100 replications, we record the mean square errors (MSE) as a performance criterion for comparison between the robust local estimates, 
given as 
\begin{eqnarray*}
\text{MSE} = \frac{\sum_{i=1}^{100} \norm{\hat{\vc{m}}}^{2}}{100}.
\end{eqnarray*}

The results are shown in Tables~\ref{tltad_2},~\ref{tltad_3} and~\ref{tltad_7}. In Table~\ref{tltad_2}, all three estimators use half sample coverage $h = 50\%$. 
We observe that the performance of the new approach LP-LTAD is quite competitive compared to the MILP-LTAD and the LTAD.
\begin{table}[h!]
\centering
\begin{tabular}{c c c c c c c c c c}
& & LP-LTAD & MILP-LTAD & LTAD  \\ \hline
&$\epsilon$& 50$\%$ &  50$\%$ & 50$\%$  \\
\hline
\multirow
{3}{*}{$p=1$}
& 0$\%$  & 0.0009  & 0.0010  & 0.0008  \\
& 20$\%$ & 0.0015  & 0.0030  & 0.0014  \\
& 40$\%$ & 0.0069  & 0.0068  & 0.0076 \\ \hline
\multirow
{3}{*}{$p=2$}
& 0$\%$   & 0.0010 & 0.0015 & 0.0014    \\
& 20$\%$  & 0.0300 & 0.0358 & 0.0311   \\
& 40$\%$  & 0.0387 & 0.0450 & 0.0323 \\ \hline
\multirow
{3}{*}{$p=3$}
& 0$\%$   & 0.0091 & 0.0094 & 0.0024   \\
& 20$\%$  & 0.0410 & 0.0401 & 0.0801   \\
& 40$\%$  & 0.0415 & 0.0858 & 0.1091  \\ \hline
\end{tabular}
\caption{MSE of local estimates, $n=50$, $p=1,2,3$}\label{tltad_2}
\end{table}
In Table~\ref{tltad_3} the proposed models LP-LTAD and MILP-LTAD use as coverage $h=20\%$, 
while the LTAD  estimator  $h=50\%$ which is the smallest that we can use in this case. 
\begin{table}[h!]
\centering
\begin{tabular}{c c c c c c c c c c}
& & LP-LTAD & MILP-LTAD & LTAD   \\ \hline
&$\epsilon$& 20$\%$ &  20$\%$ & 50$\%$  \\
\hline
\multirow{3}{*}{$p=1$}
& 0$\%$  & 0.0001 & 0.0009 &0.0008 \\
<& 20$\%$ & 0.0006 & 0.0015 & 0.0014\\
& 40$\%$ & 0.0039 & 0.0044 & 0.0076 \\ \hline
\multirow{3}{*}{$p=2$}
& 0$\%$  & 0.0009 & 0.0013 & 0.0014 \\
& 20$\%$ & 0.0244 & 0.0279 & 0.0311 \\
& 40$\%$ & 0.0281 & 0.0283 & 0.0323 \\ \hline
\multirow{3}{*}{$p=3$}
& 0$\%$  & 0.0017 & 0.0024 & 0.0024& \\
& 20$\%$ & 0.0101 & 0.0303 & 0.0801& \\
& 40$\%$ & 0.0105 & 0.0309 & 0.1091& \\ \hline
\end{tabular}
\caption{MSE of local estimates, $n=50$, $p=1,2,3$} \label{tltad_3}
\end{table}
Note  that the new estimators have improved their performance, and the LP-LTAD outperforms all. 
The results in Table~\ref{tltad_7} reveal that when the sample size increases all the estimators have similar performance. 
\begin{table}[h!]
\centering
\begin{tabular}{c c c c c c c c c c c c}
& & LP-LTAD & MILP-LTAD & LTAD  \\ \hline
&$\epsilon$ & 20$\%$ &  20$\%$ & 50$\%$  \\
\hline
\multirow
{3}{*}{$p=1$}
&0$\%$  & 0.0001 & 0.0001  & 0.0007 \\
&20$\%$ & 0.0004 & 0.0004  & 0.0012 \\
&40$\%$ & 0.0024 & 0.0023  & 0.0031 \\ \hline
\multirow
{3}{*}{$p=3$}
&0$\%$  & 0.0008 & 0.0010  & 0.0014 \\
&20$\%$ & 0.0081 & 0.0090  & 0.0805 \\
&40$\%$ & 0.0151 & 0.0199  & 0.0954 \\ \hline
\multirow
{3}{*}{$p=5$}
&0$\%$  & 0.0015 & 0.0025  & 0.0084 \\
&20$\%$ & 0.0094 & 0.0109  & 0.0906 \\
&40$\%$ & 0.0171 & 0.0210  & 0.1121 \\ \hline
\end{tabular}
\caption{MSE of local estimates, $n=100$, $p=1,3,5$} \label{tltad_7}
\end{table}

In order to investigate the effect of the presence of a correlation structure within the simulated data on the performance of the algorithms, we use a covariance 
matrix $P'$ for data generation, with elements 1 in the diagonal, and numbers $\rho$ as off-diagonal elements. We prefer a value of $\rho=0.70$ among the different simulation scenarios. Similar structures for simulated correlation data have been proposed by \cite{doi:10.1198/004017002188618509} and \cite{Hubert10adeterministic}.
The results are shown in Tables~\ref{tltad_5},~\ref{tltad_4} and~\ref{tltad_8}. We observe that the effect of the correlation does not influence the performance of the LP-LTAD estimator which outperforms the other two. 
It should be noted that the reduction of coverage from 50$\%$ to 20$\%$ does not improve the performance of the LP-LTAD.
\begin{table}[h!]
\centering
\begin{tabular}{c c c c c c c c c c}
\hline
& & LP-LTAD & MILP-LTAD & LTAD   \\ \hline
&$\epsilon$ & 50$\%$ &  50$\%$ & 50$\%$  \\
\hline
\multirow{3}{*}{$p=2$}
& 0$\%$  & 0.0005 & 0.0064  & 0.0026 \\
& 20$\%$ & 0.0041 & 0.0693  & 0.0138 \\
& 40$\%$ & 0.0376 & 0.1965  & 0.1262 \\ \hline
\multirow{3}{*}{$p=3$}
& 0$\%$  & 0.0009 & 0.0626  & 0.0199  \\
& 20$\%$ & 0.0213 & 0.1363  & 0.6115  \\
& 40$\%$ & 0.0773 & 0.2545  & 0.7672 \\ \hline
\end{tabular}
\caption{MSE of local estimates, $n=50$, correlation $\rho$=0.7, $p=2,3$} \label{tltad_5}
\end{table}
\begin{table}[h!]
\centering
\begin{tabular}{c c c c c c c c c c}
& & LP-LTAD & MILP-LTAD & LTAD   \\ \hline
&$\epsilon$& 20$\%$ &  20$\%$ & 50$\%$   \\
\hline
\multirow{3}{*}{$p=2$}
& 0$\%$  & 0.0004 & 0.0014 & 0.0026  \\
& 20$\%$ & 0.0030 & 0.0141 & 0.0138  \\
& 40$\%$ & 0.0164 & 0.1203 & 0.1262 \\ \hline
\multirow{3}{*}{$p=3$}
& 0$\%$  & 0.0005 & 0.0037  & 0.0199   \\
& 20$\%$ & 0.0056 & 0.1085  & 0.6115   \\
& 40$\%$ & 0.0225 & 0.2268  & 0.7672  \\ \hline
\end{tabular}
\caption{MSE of local estimates, $n=50$, correlation $\rho$=0.7, $p=2,3$} \label{tltad_4}
\end{table}
\begin{table}[h!]
\centering
\begin{tabular}{c c c c c c c c c c c c}
& & LP-LTAD & MILP-LTAD & LTAD  \\ \hline
&$\epsilon$& 20$\%$ &  20$\%$ & 50$\%$ \\
\hline
\multirow
{3}{*}{$p=3$}
&0$\%$  & 0.0009 & 0.0012  & 0.0094  \\
&20$\%$ & 0.0110 & 0.0210  & 0.0994  \\
&40$\%$ & 0.0215 & 0.0216  & 0.1314 \\ \hline
\multirow
{3}{*}{$p=5$}
&0$\%$  & 0.0019 & 0.0019  & 0.0105 \\
&20$\%$ & 0.0201 & 0.0301  & 0.1204 \\
&40$\%$ & 0.0274 & 0.0325  & 0.1904 \\ \hline
\end{tabular}
\caption{MSE of local estimates, $n=100$, correlation $\rho=0.7$, $p=3,5$} \label{tltad_8}
\end{table}

To illustrate the performance of the estimators on data contaminated with \textit{intermediate} outliers, we replaced 20\% or 40\% of the first rows of the 
multivariate sample, $N_p(\vc{0},\vc{I})$, with intermediate outliers from a multivariate normal distribution $N_p(0.75,0.5)$, as suggested by~\cite{roel2009}. 
We prefer to reduce the coverage to $h=20\%$, among the different scenarios $30\%,40\%,50\%$, which enables the new LP approach to identify the 
intermediate outliers. In the  results given by Tables~\ref{tltad_6} and~\ref{tltad_9} it is evident that the LP-LTAD outperforms the other estimators. 
This is especially true for heavily contaminated data.
\begin{table}[h!]
\centering
\begin{tabular}{c c c c c c c c c c}
& & LP-LTAD & MILP-LTAD & LTAD  \\ \hline
&$\epsilon$& 20$\%$ &  20$\%$ & 50$\%$  \\
\hline
\multirow{3}{*}{$p=1$}
& 0$\%$  & 0.0001 & 0.0007 & 0.0009 \\
& 20$\%$ & 0.0001 & 0.0421 & 0.0411\\
& 40$\%$ & 0.0018 & 0.0621 & 0.1429\\ \hline
\multirow{3}{*}{$p=2$}
& 0$\%$  & 0.0006 & 0.0426 & 0.0451\\
& 20$\%$ & 0.0391 & 0.1146 & 0.3363\\
& 40$\%$ & 0.0415 & 0.1791 & 0.3946\\ \hline
\multirow{3}{*}{$p=3$}
& 0$\%$  & 0.0184 & 0.0689 & 0.0934\\
& 20$\%$ & 0.1120 & 0.3442 & 0.3498 \\
& 40$\%$ & 0.1230 & 0.3562 & 0.3765\\ \hline
\end{tabular}
\caption{MSE of local estimates, $n=50$, $p=1,2,3$} \label{tltad_6}
\end{table}

\begin{table}[h!]
\centering
\begin{tabular}{c c c c c c c c c c }
& & LP-LTAD & MILP-LTAD & LTAD  \\ \hline
&$\epsilon$& 20$\%$ &  20$\%$ & 50$\%$  \\
\hline
\multirow{3}{*}{$p=1$}
& 0$\%$  & 0.0001 & 0.0001 & 0.0001 \\
& 20$\%$ & 0.0002 & 0.0211 & 0.0345\\
& 40$\%$ & 0.0009 & 0.0321 & 0.1425\\ \hline
\multirow{3}{*}{$p=3$}
& 0$\%$  & 0.0009 & 0.0091 & 0.0101\\
& 20$\%$ & 0.0315 & 0.0846 & 0.2214\\
& 40$\%$ & 0.0318 & 0.0855 & 0.2319\\ \hline
\multirow{3}{*}{$p=5$}
& 0$\%$  & 0.0009 & 0.0102 & 0.0798\\
& 20$\%$ & 0.0517 & 0.1054 & 0.2421\\
& 40$\%$ & 0.0518 & 0.1094 & 0.2541\\ \hline
\end{tabular}
\caption{MSE of local estimates, $n=100$, $p=1,3,5$} \label{tltad_9}
\end{table}

Finally we generate a large data set with  $n=500$ and $p=10,20$ with the same contamination and distributions as the previous simulations. 
The results are shown in Tables~\ref{tltad_10} and~\ref{tltad_11}. 
\begin{table}[h!]
\centering
\begin{tabular}{c c c c c c c c c c c c}
& & LP-LTAD & MILP-LTAD & LTAD   \\ \hline
&$\epsilon$& 20$\%$ &  20$\%$ & 50$\%$  \\
\hline
\multirow
{3}{*}{$p=10$}
&0$\%$  & 0.0098 & 0.0099  & 0.0167  \\
&20$\%$ & 0.0109 & 0.0111  & 0.0977 \\
&40$\%$ & 0.0109 & 0.0112  & 0.0999 \\ \hline
\multirow
{3}{*}{$p=20$}
&0$\%$  & 0.0104 & 0.0110  & 0.0198  \\
&20$\%$ & 0.0121 & 0.0134  & 0.1112 \\
&40$\%$ & 0.0124 & 0.0139  & 0.1123 \\ \hline
\end{tabular}
\caption{MSE of local estimates, $n=500$, $p=10$} \label{tltad_10}
\end{table}
\begin{table}[h!]
\centering
\begin{tabular}{c c c c c c c c c c}
& & LP-LTAD & MILP-LTAD & LTAD  \\ \hline
&$\epsilon$& 20$\%$ &  20$\%$ & 50$\%$  \\
\hline
\multirow{3}{*}{$p=10$}
& 0$\%$  & 0.0091 & 0.0098 & 0.0107 \\
& 20$\%$ & 0.0108 & 0.0112 & 0.3155 \\
& 40$\%$ & 0.0108 & 0.0112 & 0.3385 \\ \hline
\multirow{3}{*}{$p=20$}
& 0$\%$  & 0.0101 & 0.0101 & 0.0110 \\
& 20$\%$ & 0.0131 & 0.0142 & 0.4514 \\
& 40$\%$ & 0.0132 & 0.0143 & 0.4919 \\ \hline
\end{tabular}
\caption{MSE of local estimates, $n=500$, $p=20$} \label{tltad_11}
\end{table}

In summary,  based on the computational results we can conclude that  there are negligible differences between the three estimators for non-correlated data and 
contaminated with strong outliers. In the case of correlated data LP-LTAD has the best performance. Also, if the data is contaminated with intermediate 
outliers the LP-LTAD has is superior because it can work with coverage less than $50\%$ while, on the other hand, the LTAD includes some of the intermediate 
outliers into the coverage set so they become masked.

\section{Conclusions} \label{sec_conclusions}
In this work, we develop numerical methods for computing  the multivariate LTAD estimator based on the $L^1$ norm,  by reformulating its original mixed integer 
nonlinear formulation. 
We show that the MINLP is equivalent to an MILP and subsequently to an LP under some conditions on the location estimate.  An LP-based iterative approach 
is then developed for computing the estimator, by transforming the data and solving the resulting linear programs by subgradient optimization. 
The new LP-LTAD formulation can also be viewed as a new trimming procedure that trims away large residuals implicitly by shrinking the associated observations to zero. 
The new approach yields a robust location estimate without loosing efficiency. We perform numerical experiments and show that the new estimate performs well 
even in the case of contaminated and correlated multivariate data. The LP-LTAD procedure can be used when the data involves both type of outliers, strong and 
intermediate,  and also when the coverage is smaller than half the sample observations.

\bibliographystyle{plain}

\begin{thebibliography}{10}


\bibitem{Aggar:2013}
Charu C. Aggarwal.
\newblock Outlier detection in graphs and networks.
\newblock In {\em Outlier Analysis}, pages 343--371. Springer New York, 2013.


\bibitem{Aggar:2011}
Charu C. Aggarwal, Yuchen Zhao and Philip S. Yu.
\newblock Outlier Detection in Graph Streams.
\newblock In {\em Proceedings of the 2011 IEEE 27th International Conference on Data Engineering}, pages 399--409. IEEE Computer
Society, 2011.


\bibitem{Bassett:1991}
G.W. Bassett.
\newblock Equivariant, monotonic, 50\% breakdown estimators.
\newblock {\em The American Statistician}, 45(2):135--137, 1991.


\bibitem{chakraborty1998}
Biman Chakraborty, Probal Chaudhuri, and Hannu Oja.
\newblock Operating transformation retransformation on spatial median and angle
  test.
\newblock \emph{Statistica Sinica}, 8:\penalty0 767--784, 1998.


\bibitem{ChaPitZiou:2015}
Christos Chatzinakos, Leonidas Pitsoulis, and George Zioutas.
\newblock Optimization techniques for robust multivariate location and scatter estimation
\newblock \emph{Journal of Combinatorial Optimization}, \penalty0 online, \penalty0 2015.


\bibitem{Filzmoser20081694}
Peter Filzmoser, Ricardo Maronna, and Mark Werner.
\newblock Outlier identification in high dimensions.
\newblock \emph{Computational Statistics \& Data Analysis}, 52\penalty0
  (3):\penalty0 1694 -- 1711, 2008.

\bibitem{Hettmansperger2002}
Thomas~P. Hettmansperger.
\newblock A practical affine equivariant multivariate median.
\newblock \emph{Biometrika}, 89\penalty0 (4):\penalty0 851--860, 2002.

\bibitem{Hubert10adeterministic}
Mia Hubert, Peter~J. Rousseeuw, and Tim Verdonck.
\newblock A deterministic algorithm for the mcd, 2010.

\bibitem{doi:10.1198/004017002188618509}
Ricardo~A Maronna and Ruben~H Zamar.
\newblock Robust estimates of location and dispersion for high-dimensional
  datasets.
\newblock \emph{Technometrics}, 44\penalty0 (4):\penalty0 307--317, 2002.

\bibitem{MiGuEll:03}
G.~Mitra, M.~Guertler, and F.~Ellison.
\newblock Algorithms for the solution of large-scale quadratic mixed integer
  programming (qmip) models.
\newblock In \emph{International Symposium in Mathematical Programming}, 2003.

\bibitem{Neykov2012}
N.M. Neykov, P.~\v{C}\'{i}\v{z}ek, P.~Filzmoser, and P.N. Neytchev.
\newblock The least trimmed quantile regression.
\newblock \emph{Computational Statistics \& Data Analysis}, 56\penalty0
  (6):\penalty0 1757--1770, 2012.

\bibitem{Oja1983}
Hannu Oja.
\newblock Descriptive statistics for multivariate distributions.
\newblock \emph{Statistics \& Probability Letters}, 1\penalty0 (6):\penalty0
  327--332, 1983.

\bibitem{Stefan}
Ella Roelant and StefanVan Aelst.
\newblock An l1-type estimator of multivariate location and shape.
\newblock \emph{Statistical Methods and Applications}, 15\penalty0
  (3):\penalty0 381--393, 2007.

\bibitem{roel2009}
Ella Roelant, Stefan Aelst, and Gert Willems.
\newblock The minimum weighted covariance determinant estimator.
\newblock \emph{Metrika}, 70\penalty0 (2):\penalty0 177--204, 2009.

\bibitem{Rousseeuw1985}
Peter~J. Rousseeuw.
\newblock Multivariate estimation with high breakdown point.
\newblock \emph{Mathematical Statistics and Applications}, B:\penalty0
  283--297, 1985.

\bibitem{Rousseeuw1998}
Peter~J. Rousseeuw and Katrien~Van Driessen.
\newblock A fast algorithm for the minimum covariance determinant estimator.
\newblock \emph{Technometrics}, 41:\penalty0 212--223, 1998.

\bibitem{rousseeuw1}
P.~J. Rousseeuw.
\newblock Least median of squares regression.
\newblock {\em Journal of the American Statistical Association}, 79:871--881,
  1984.

\bibitem{Tableman1994}
Mara Tableman.
\newblock The asymptotics of the least trimmed absolute deviations (LTAD)
  estimator.
\newblock \emph{Statistics \& Probability Letters}, 19\penalty0 (5):\penalty0
  387--398, 1994.

\end{thebibliography}

\end{document}